\documentclass{article}
\newcommand{\B}{{\cal B}}
\newcommand{\D}{{\cal D}}
\newcommand{\E}{{\cal E}}
\newcommand{\C}{{\cal C}}
\newcommand{\F}{{\cal F}}
\newcommand{\A}{{\cal A}}
\newcommand{\Hh}{{\cal H}}
\newcommand{\Pp}{{\cal P}}
\newcommand{\G}{{\cal G}}
\newcommand{\Z}{{\bf Z}}
\newtheorem{theorem}{Theorem}[section]
\newtheorem{definition}{Definition}[section]
\newtheorem{lemma}[theorem]{Lemma}
\newtheorem{corollary}[theorem]{Corollary}

\newtheorem{remark}[theorem]{Remark}

\def\whitebox{{\hbox{\hskip 1pt
 \vrule height 6pt depth 1.5pt
 \lower 1.5pt\vbox to 7.5pt{\hrule width
    3.2pt\vfill\hrule width 3.2pt}%
 \vrule height 6pt depth 1.5pt
 \hskip 1pt } }}
\def\qed{\ifhmode\allowbreak\else\nobreak\fi\hfill\quad\nobreak
     \whitebox\medbreak}
\newcommand{\proof}{\noindent{\it Proof.}\ }
\newcommand{\qedd}{\hfill$\mbox$}
\newcommand{\ignore}[1]{}
\newcommand{\tabincell}[2]{\begin{tabular}{@{}#1@{}}#2 \end{tabular}}
\usepackage{amsmath}
\usepackage{amssymb}
\usepackage{color}

\begin{document}

\medskip
\title{Partial Difference Sets with Denniston Parameters in Elementary Abelian $p$-Groups}

 \author{{\small   Jingjun Bao$^{1}$, \ Qing Xiang$^2$, \ Meng Zhao$^2$} \\
 {\small $^1$ Department of Mathematics, Ningbo University, Ningbo 315211, China}\\
 {\small $^2$ Department of Mathematics and Shenzhen International Center for Mathematics,} \\ {\small Southern University of Science and Technology, Shenzhen 518055, China}\\
 {\small E-mail:  \ baojingjun@hotmail.com; xiangq@sustech.edu.cn; zhaom@mail.sustech.edu.cn}\\
% {\small E-mail: baojingjun@hotmail.com}\\
 }

%
%{\bf SHORT RUNNING HEAD:}  Steiner Quadruple Systems
\date{}
\maketitle

\begin{abstract}
\medskip
%Denniston showed that the necessary condition is sufficient for the existence of a maximal arc in PG$(2, 2^n)$. 
Denniston \cite{D1969} constructed partial difference sets (PDS) with parameters $(2^{3m}, (2^{m+r}-2^m+2^r)(2^m-1), 2^m-2^r+(2^{m+r}-2^m+2^r)(2^r-2), (2^{m+r}-2^m+2^r)(2^r-1))$ in elementary abelian groups of order $2^{3m}$ for all $m\geq 2$ and $1 \leq r < m$. These PDS correspond to maximal arcs in the Desarguesian projective planes PG$(2, 2^m)$. Davis et al. \cite{DHJP2024} and also De Winter \cite{dewinter23} presented constructions of  PDS  with Denniston parameters $(p^{3m}, (p^{m+r}-p^m+p^r)(p^m-1), p^m-p^r+(p^{m+r}-p^m+p^r)(p^r-2), (p^{m+r}-p^m+p^r)(p^r-1))$ in elementary abelian groups of order $p^{3m}$ for all $m \geq 2$ and  $r \in \{1, m-1\}$, where $p$ is an odd prime. The constructions in \cite{DHJP2024, dewinter23}  are particularly intriguing, %The construction is interesting, 
as it was shown by Ball, Blokhuis, and Mazzocca \cite{BBM1997}  that no nontrivial maximal arcs in PG$(2, q^m)$ exist for any odd prime power $q$. In this paper, %we settle an open problem about the existence of PDSs with the Denniston parameters which was proposed by Davis et al. \cite{DHJP2024}. 
we show that PDS with Denniston parameters $(q^{3m}, (q^{m+r}-q^m+q^r)(q^m-1), q^m-q^r+(q^{m+r}-q^m+q^r)(q^r-2), (q^{m+r}-q^m+q^r)(q^r-1))$ exist in elementary abelian groups of order $q^{3m}$ for all $m \geq 2$ and $1 \leq r < m$,  where $q$ is an arbitrary  prime power. 

\vspace{0.1in}


\smallskip
\end{abstract}
{
\section{Introduction}

Let $G$ be a finite multiplicative group of order $v$ with identity $e$. A $k$-subset $D$ of $G$ is called a $(v, k, \lambda, \mu)$-{\em partial difference set} (PDS) if  the list of ``differences" $d_1d_2^{-1}, d_1,d_2\in D, d_1\neq d_2 $, represents each nonidentity element of $D$ exactly $\lambda$ times and each nonidentity element of $G\setminus D$ exactly $\mu$ times. PDS are often studied within the context of the group ring $\mathbb{Z}[G]$. For a subset $D$ of a finite group $G$, we still use $D$ to denote the group ring element $\sum\limits_{d \in D} d\in \mathbb{Z}[G]$, and define $D^{(-1)}=\sum\limits_{d \in D} d^{-1}$.
%we write $D=\sum\limits_{d \in D} d$ and $D^{(-1)}=\sum\limits_{d \in D} d^{-1}$.  
Then $D$ is a $(v, k, \lambda, \mu)$-PDS if and only if it satisfies the following equation in $\mathbb{Z}[G]$:
$$DD^{(-1)}=\mu G+(\lambda-\mu)D+\gamma e$$
where $\gamma=k-\mu$ if $e\notin D$ and $\gamma=k-\lambda$ if $e\in D$. A $(v, k, \lambda, \mu)$-PDS with $\lambda=\mu$ is simply a $(v,k,\lambda)$-difference set. If $D$ is a $(v, k, \lambda, \mu)$-PDS with $\lambda\neq \mu$, then $D^{(-1)}=D$ and 
$$D^2=\mu G+(\lambda-\mu)D+\gamma e.$$

A well-known example of PDS is the Paley PDS. Let $\mathbb{F}_q$ be the  finite field with $q$ elements, where $q$ is a prime power congruent to 1 modulo  $4$. Then the set of nonzero squares in $\mathbb{F}_q$ forms a $(q, \frac{q-1}{2}, \frac{q-5}{4}, \frac{q-1}{4})$-PDS in the additive group of $\mathbb{F}_q$ and this PDS is usually called the  {\em Paley} PDS. 

A $(v, k, \lambda, \mu)$-PDS is said to be of Latin square type (resp. negative Latin square type) if 
$$(v,k,\lambda,\mu)=(n^2, r(n-\epsilon), \epsilon n+r^2-3\epsilon r, r^2-\epsilon r)$$
and $\epsilon =1$ (resp. $\epsilon =-1$).  Many constructions for PDS of Latin square type are known, see,  for example, \cite{CK1986, CRX1996, LM1990, M1994}.  %Latin square and negative Latin square type PDSs have connections with amorphic association schemes (See \cite{V2003, V2001}). %PDSs in elementary abelian groups are equivalent to projective two weight codes \cite{CK1986}. 

If $e\notin D$ and $D^{(-1)}=D$, then the PDS $D$ is said to be  {\em regular}. Regular PDS are equivalent to strongly regular Cayley graphs \cite{M1994}. PDS in elementary abelian groups are closely related to projective two-weight codes and two-intersection sets in projective spaces over finite fields \cite{BM2022, CK1986}. We now explain the connection between PDS and projective two-intersection sets.

Let $p$ be a prime and let $q$ be a power of prime $p$. %$q=p^m$. 
The Desarguesian $(m-1)$-dimensional projective space over $\mathbb{F}_q$ is denoted by PG$(m-1,q)$. The vector space underlying this projective space is denoted by $V_m(q)$, which is an $m$-dimensional vector space over $\mathbb{F}_q$. A {\em projective} $(n, m, h_1, h_2)$ {\em set} ${\cal O}$ is a proper, nonempty subset of $n$ points of PG$(m-1,q)$ with the property that every hyperplane meets ${\cal O}$ in $h_1$ points or $h_2$ points. Many examples of projective $(n,m,h_1,h_2)$ sets are known. For example, a hyperoval of PG$(2,2^m)$ is a projective $(2^m+2, 3, 0, 2)$ set in PG$(2,2^m)$, and a unital of PG$(2,q^2)$ is a projective $(q^3+1, 3, 1, q+1)$ set in PG$(2,q^2)$.

Let ${\cal O}=\{\langle y_1\rangle, \langle y_2\rangle, \ldots, \langle y_n\rangle \}$ be a set consisting of  $n$ points of PG$(m-1,q)$.  Define $\Omega=\{v\in V_m(q)~|~\langle v\rangle \in {\cal O} \}$, which is the set of nonzero vectors in $V_m(q)$ corresponding to ${\cal O}$. That is, $\Omega=\mathbb{F}_q^\ast {\cal O}$, where $\mathbb{F}_q^{\ast}=\mathbb{F}_q\setminus\{0\}$; thus $\Omega$ is $\mathbb{F}_q^*$-invariant (i.e., $\alpha \Omega=\Omega$, $\forall \alpha\in \mathbb{F}_q^*$). We have the following lemma. 

\begin{lemma}
Let ${\cal O}$ and $\Omega$ be defined as above. Then ${\cal O}$ is a projective $(n,m,h_1,h_2)$ set in {\rm PG}$(m-1,q)$ if and only if $\Omega$ is a $(q^m, (q-1)n, \lambda, \mu)$-PDS in the elementary abelian group $(V_m(q),+)$, where $\lambda=(q-1)n+(qh_1-n)(qh_2-n)+q(h_1+h_2)-2n$ and $\mu=(q-1)n+(qh_1-n)(qh_2-n)$.
\end{lemma}

Let $n\ge 2, d\ge 1$ be integers. An $(n,d)$-{\em arc} in PG$(2, q)$ is a set of $n$ points, of which no $d+1$ points are collinear. Let ${\cal K}$ be an $(n,d)$-arc in PG$(2, q)$, and let $x$ be a point of ${\cal K}$. Then each of the $q+1$ lines  passing through $x$ contains at most $d-1$ points of ${\cal K}$. Therefore, 
$$n\leq 1+(q+1)(d-1).$$
An $(n,d)$-arc is called {\em maximal}\footnote{Van Lint and Wilson \cite{vlWilson} proposed to use ``perfect" instead of ``maximal", but the use of perfect arc never caught on.}  if $n=1+(q+1)(d-1)$. Any line of PG$(2, q)$ that contains a point of a maximal arc ${\cal K}$ evidently contains exactly $d$ points of that arc, that is, 
$$\left|L\cap {\cal K}\right|=0\ {\rm or}\ d,$$
for any line $L$ of PG$(2, q)$. Hence, a maximal $(n,d)$-arc ${\cal K}$ in PG$(2, q)$ is a projective $(n, 3, 0, d)$ set in PG$(2, q)$. Furthermore, it can be shown that if ${\cal K}$ is a maximal $(n,d)$-arc in PG$(2, q)$, then $d$ divides $q$. For $q=2^m$ and for every $1\le r<m$,  Denniston \cite{D1969} constructed a maximal $(2^{m+r}-2^m+2^r, 2^r)$-arc in PG$(2, 2^m)$. Therefore, by Lemma~1.1, there exists an $\mathbb{F}_{2^m}^*$-invariant PDS with parameters $(2^{3m}, (2^{m+r}-2^m+2^r)(2^m-1), 2^m-2^r+(2^{m+r}-2^m+2^r)(2^r-2), (2^{m+r}-2^m+2^r)(2^r-1))$ in elementary abelian groups $\mathbb{Z}_{2}^{3m}$ for all $m\geq 2$ and $1 \leq r < m$.

%Davis and Xiang \cite{DX2000} used Galois Rings to construct PDSs with the Denniston parameters in abelian group $\mathbb{Z}_{4}^{m}\times \mathbb{Z}_{2}^{m}$ for all $m\geq 2$ and $r\in \{1,m-1\}$. Brady \cite{B2022} employed a computer search to find %examples in the $m=2, r=1$ case for $51$ groups of order $64$. %, and Smith [??] extended this to an exhaustive search that found examples in $73$ out of $267$ nonisomorphic groups of order $64$. 

For any odd prime power $q$, Ball et al. \cite{BBM1997} showed that no nontrivial maximal arcs exist in PG$(2, q)$. It follows that for any odd prime power $q$, there does not exist an $\mathbb{F}_q^*$-invariant PDS with Denniston parameters in the additive group of $\mathbb{F}_q^3$.

So it came as a surprise when Davis et al. \cite{DHJP2024} and also De Winter \cite{dewinter23} could construct PDS with Denniston parameters $(p^{3m}, (p^{m+r}-p^m+p^r)(p^m-1), p^m-p^r+(p^{m+r}-p^m+p^r)(p^r-2), (p^{m+r}-p^m+p^r)(p^r-1))$ in elementary abelian groups of order $p^{3m}$ for all $m \geq 2$ and $r \in \{1, m-1\}$, where $p$ is an odd prime. Indeed, the PDS constructed in \cite{DHJP2024, dewinter23} are {\bf not} $\mathbb{F}_{p^m}^*$-invariant; so the constructive result of Davis et al. in \cite{DHJP2024} and De Winter \cite{dewinter23} does not contradict the nonexistence result of Ball et al. \cite{BBM1997}.  In \cite{DHJP2024}, the authors asked whether it is possible to construct PDS in elementary abelian $p$-groups of order $p^{3m}$ ($p$ an odd prime) with Denniston parameters for $2\le r\le m-2$. In this paper, we answer this question in the affirmative by presenting a cyclotomic construction of PDS with Denniston parameters in the additive group of $\mathbb{F}_{q^{m}}\times \mathbb{F}_{q^{2m}}$ for all $m\geq 2$ and $1 \leq r < m$, where $q$ is an arbitrary prime power.

%The paper is organized as follows. Section 2 introduces quadratic forms and reformulates the point set of a nondegenerate elliptic quadratic forms in terms of cyclotomic classes. Section 3 presents a construction of PDSs with the Denniston parameters in abelian group $\mathbb{F}_{q}^{3m}$ for all $m\geq 2, 1 \leq r < m$. 

\section{Characters and quadratic forms}

We start by giving a quick introduction to characters of abelian groups. Let $G$ be a finite abelian group. A (complex) character 
$\chi$ of $G$ is a homomorphism from $G$ to $\mathbb{C}^*$. A character $\chi$ of $G$ is called {\em principal} if $\chi(g)=1$ for all $g\in G$; otherwise it is called {\em nonprincipal}. Let $\chi$ be a character of $G$ and $A=\sum\limits_{g\in G}a_gg \in \mathbb{C}[G].$ We define $\chi(A)=\sum\limits_{g\in G}a_g\chi(g).$ 

%Let $p$ be an odd prime, $f$ a positive integer, and $q=p^f$.
%Let $q$ be a power of prime $p$ and let $m$ be a positive integer. 
In the rest of this paper, let $p$ be a prime (even or odd) and let $q$ be a power of prime $p$. Consider the finite field $\mathbb{F}_{q^m}$. Let $\xi_p$ be a fixed complex primitive $p^{\rm th}$ root of unity and let Tr$_{q^m/p}$ be the trace function from $\mathbb{F}_{q^m}$ to $\mathbb{F}_p$. The {\em canonical additive character} $\psi_{\mathbb{F}_{q^m}}$ of $\mathbb{F}_{q^m}$ is defined as
%Let $\psi_{\mathbb{F}_{q^m}}$ be the canonical additive character of $\mathbb{F}_{q^m}$, that is, the character of the additive group of $\mathbb{F}_{q^m}$ defined by
\begin{equation}
\label{char}
\psi_{\mathbb{F}_{q^m}}(x)=\xi_p^{{\rm Tr}_{q^m/p}(x)},\; \forall x\in \mathbb{F}_{q^m}.
\end{equation}
All the characters of the additive group of $\mathbb{F}_{q^m}$ are then given by $\psi_a$,  where $a$ runs through all elements of $\mathbb{F}_{q^m}$. Here $\psi_a$ is the character defined by 
$$\psi_a(x)=\psi_{\mathbb{F}_{q^m}}(ax),\;\forall x\in \mathbb{F}_{q^m}.$$ 
Let ${\widehat{\mathbb{F}}_{q^m}}$ denotes the character group of the additive group of $\mathbb{F}_{q^m}$. In exactly the same way, all characters of the additive group of $\mathbb{F}_{q^{2m}}$ are given by $\varphi_b$, where $b$ runs through the elements of $\mathbb{F}_{q^{2m}}$. Here $\varphi_b$ is defined by $\varphi_b(x)=\xi_p^{{\rm Tr}_{q^{2m}/p}(bx)}$, $\forall x\in \mathbb{F}_{q^m}$. 

%Let $\psi_a(x)=\psi_{\mathbb{F}_{q^m}}(ax)\;{\rm for}\;x\in \mathbb{F}_{q^m},$ which defines an additive character of $\mathbb{F}_{q^m}$. Moreover, as $a$ runs through  $\mathbb{F}_{q^m}$, we obtain all additive character of $\mathbb{F}_{q^m}$. Denote the set of all additive character of $\mathbb{F}_{q^m}$ by ${\widehat{\mathbb{F}}_{q^m}}$, which indeed forms a group. 

%Similarly, let $\varphi_b(x)=\xi_p^{{\rm Tr}_{q^{2m}/p}(bx)}$ be an additive character of $\mathbb{F}_{q^{2m}}$, where $b\in\mathbb{F}_{q^{2m}}$. It follows that ${\widehat{\mathbb{F}}_{q^{2m}}}=\{\varphi_b|b\in\mathbb{F}_{q^{2m}}\}$.

%All the characters of the additive group of $\mathbb{F}_{q^m}$ are $\psi_a$,  where $a$ runs through all elements of $\mathbb{F}_{q^m}$. Here $\psi_a$ is the character defined by
% $$\psi_a(x)=\psi_{\mathbb{F}_{q^m}}(ax)\;{\rm for}\;x\in \mathbb{F}_{q^m}.$$ 
% Let ${\widehat{\mathbb{F}}_{q^m}}$ denotes the character group of the additive group of $\mathbb{F}_{q^m}$. In exactly the same way, all characters of the additive group of $\mathbb{F}_{q^{2m}}$ are given by $\varphi_b$, where $b$ runs through the elements of $\mathbb{F}_{q^{2m}}$. Here $\varphi_b$ is defined by $\varphi_b(x)=\xi_p^{{\rm Tr}_{q^{2m}/p}(bx)}$, $\forall x\in \mathbb{F}_{q^m}$. 

%In the case when $G$ is an abelian group, using the Fourier inversion formula, we have the following spectral characterization of PDS (See \cite{M1994}).

In the case when $G$ is an abelian group, the following theorem gives a spectral characterization of PDS (see \cite{M1994}). 
%Character theory has been one of the primary tools used to investigate PDSs because of the following theorem (See).
\begin{theorem}\label{PDS}
Let $G$ be a finite abelian group of order $v$ and let $D$ be a $k$-subset of $G$ such that $e\not\in D$ and $D^{(-1)}=D$. Suppose $k,\lambda, \mu$ are positive integers such that $k^2=\mu v+(\lambda -\mu)k+(k-\mu)$. Then $D$ is a 
$(v,k,\lambda,\mu)$-PDS in $G$ if and only if for every nonprincipal character $\psi$ of $G$, we have 
$$\psi(D)=\frac{(\lambda-\mu)\pm \sqrt{(\mu-\lambda)^2+4(k-\mu)}}{2}.$$
\end{theorem}

Let ${\cal O}=\{\langle y_1\rangle, \langle y_2\rangle, \ldots, \langle y_n\rangle \}$ be a set consisting  of $n$ points of PG$(m-1,q)$. Let $\Omega=\{v\in \mathbb{F}_{q^m}~|~\langle v\rangle \in {\cal O} \}$ be the set of nonzero vectors in $\mathbb{F}_{q^m}$ corresponding to ${\cal O}$. For any nonprincipal additive character $\psi_a$ of $\mathbb{F}_{q^{m}}$, we have 
\[
\begin{array}{l}
\vspace{0.2cm}\psi_a(\Omega)=(q-1)|a^{\perp}\cap \{y_1,y_2,\ldots,y_n\}|+(-1)(n-|a^{\perp}\cap \{y_1,y_2,\ldots,y_n\}|)\\
\vspace{0.2cm}\hspace{1cm}=q|a^{\perp}\cap \{y_1,y_2,\ldots,y_n\}|-n,
\end{array}
\]
where $a^{\perp}=\{y \in \mathbb{F}_{q^{m}}~|~ {\rm Tr}_{q^m/q}(ay)=0\}$, and Tr$_{q^m/q}$ is the trace function from $\mathbb{F}_{q^m}$ to $\mathbb{F}_{q}$. Then we have the following lemma. 

\begin{lemma} {\em ({\rm \cite{W1997}})}
Let ${\cal O}$ and $\Omega$ be defined as above. Then ${\cal O}$ is a projective $(n, m, h_1, h_2)$ set in PG$(m-1,q)$ if and only if $\psi_a(\Omega)=qh_1-n$ or $qh_2-n$, for any $\psi_a\in {\widehat{\mathbb{F}}_{q^m}}\setminus\{\psi_0\}$.
\end{lemma}
It follows that  projective $(n,m,h_1,h_2)$ sets in PG$(m-1,q)$ and  $\mathbb{F}_q^*$-invariant PDS in $(V_m(q),+)$ are equivalent by the above lemma.

%From this lemma, we see that projective $(n,k,h_1,h_2)$ sets in PG$(k-1,q)$ and  $\mathbb{F}_q^*$-invariant PDS in $(V_k(q),+)$ are equivalent objects.

We will also need to use quadratic forms (or quadrics) in our construction of PDS.
\begin{definition}
Let $V$ be an $m$-dimensional vector space over a field $F$. A function $Q: V \rightarrow F$ is called a {\it quadratic form} if  
%(i) For $a\in F,$ $Q(ax)=a^2Q(x).$\\
%(ii) The function $R: V\times V \rightarrow F,$ $R(u,v):=Q(u+v)-Q(u)-Q(v)$ is bilinear.\\
\[
\begin{array}{l}
\vspace{0.1cm}{\rm (i)\ For}\ a\in F,\ Q(ax)=a^2Q(x).\\
\vspace{0.1cm}{\rm (ii)\ The\ function}\ B: V\times V \rightarrow F,\ B(u,v):=Q(u+v)-Q(u)-Q(v)\ {\rm is\ bilinear}.
\end{array}
\]
We say that $Q$ is {\em nonsingular} if the subspace $W$ of $V$ with the property that $Q$ vanishes on $W$ and $B(w,v)=0$ for all $v\in V$ and $w\in W$ is the zero subspace. If the field $F$ has odd characteristic, then $Q$ is nonsingular if and only if 
$B$ is nondegenerate; that is, $B(u,v)=0$ for all $v\in V$ implies $u=0.$ But this may not be true if $F$ has characteristic $2$, because in that case, $Q$ may not vanish on the radical ${\rm Rad}(V)=\{w\in V\mid B(w,v)=0,\forall v\in V\}$. However, if $V$ is an even-dimensional vector space over an even-characteristic field $F$, it is still true that $Q$ is nonsingular if and only if $B$ is nondegenerate (cf. \cite[p.14]{Cameron}).
\end{definition}

%In the rest of this paper, $q$ is a power of a prime $p$ (which can be even or odd). 
We will consider $Q: \mathbb{F}_{q^{2m}}\rightarrow \mathbb{F}_q$, defined by $Q(x)={\rm Tr}_{q^m/q}\left(x^{q^m+1}\right), \forall x\in \mathbb{F}_{q^{2m}}$. 

\begin{lemma}\label{QF} {\em ({\rm \cite[Theorem 3.2]{CLeo}})}
The function $Q :\mathbb{F}_{q^{2m}}\rightarrow \mathbb{F}_q$  defined above is a nonsingular quadratic form. Furthermore, $Q$ is of elliptic type.
\end{lemma}

%This is Theorem 3.2 in \cite{CLeo}. The reader can find a proof of Lemma~\ref{QF} in \cite{CLeo}.

It is well known the zero-set of a nonsingular quadratic form defined on an even-dimensional vector space over $\mathbb{F}_q$ is a PDS (which is referred to as Example {\bf RT}2 in \cite{CK1986}).

\begin{lemma}\label{PDS1} {\em ({\rm \cite{CK1986, M1994}})}
Let $Q :\mathbb{F}_{q^{2m}}\rightarrow \mathbb{F}_q$ be a nonsingular quadratic form. Then 
$$X=\{x \in \mathbb{F}_{q^{2m}}\setminus \{0\}|~ Q(x)=0\}$$
is a regular $(q^{2m}, (q^m-\epsilon)(q^{m-1}+\epsilon), q^{2m-2}+\epsilon q^{m-1}(q-1)-2, q^{2m-2}+\epsilon q^{m-1})$-PDS in the additive group of $\mathbb{F}_{q^{2m}}$, where $\epsilon=1$ or $-1$ according as $Q$ is hyperbolic or elliptic. In particular, for $b\in \mathbb{F}_{q^{2m}}\setminus \{0\},$
\[
\begin{array}{l}
\varphi_b(X)=
\left\{\begin{array}{ll}
\vspace{0.2cm}\epsilon q^{m-1}(q-1)-1, & {\rm  if} \ \ Q(b)=0, \\
-\epsilon q^{m-1}-1,  & {\rm otherwise}.\
\end{array}
\right .
\end{array}
\]
\end{lemma}

%Davis et al. used the nondegenerate quadratic form $Q(x)=Tr_{p^m/p}\left(x^{p^m+1}\right)$ to construct PDSs with the Denniston parameters $(p^{3m}, (p^{m+r}-p^m+p^r)(p^m-1), p^m-p^r+(p^{m+r}-p^m+p^r)(p^r-2), (p^{m+r}-p^m+p^r)(p^r-1))$ for all $m \geq 2, r \in \{1, m-1\}$. We use the nondegenerate quadratic form $Q(x)=Tr_{q^m/q}\left(x^{q^m+1}\right)$ to construct PDSs with the Denniston parameters for all $m \geq 2, 1\leq r \leq m-1$. 

\begin{definition}
Let $\alpha$ be a fixed primitive element of $\mathbb{F}_{q^{2m}}$, and let $N$ be a divisor of $q^{2m}-1$. The $N^{\rm th}$ {\em cyclotomic classes} $C_0^{(N, q^{2m})}, C_1^{(N, q^{2m})}, \ldots, C_{N-1}^{(N, q^{2m})}$ are defined by 
$$ C_i^{(N, q^{2m})}=\left\{\alpha^{i+Nj}~\big{|} ~0\leq j< \frac{q^{2m}-1}{N}\right\},$$ 
where $0\leq i<N.$
\end{definition}

The zero-set of the quadratic form $Q(x)={\rm Tr}_{q^m/q}\left(x^{q^m+1}\right)$ has the following cyclotomic description.
\begin{lemma}\label{QF1}
Let $\alpha$ be a fixed primitive element of $\mathbb{F}_{q^{2m}}$, and let $\omega=\alpha^{q^m+1}$. Let $Q :\mathbb{F}_{q^{2m}}\rightarrow \mathbb{F}_q$ be defined as above. Then 
$$\{x \in \mathbb{F}_{q^{2m}}\setminus \{0\}|~ Q(x)=0\}=\bigcup\limits_{i \in I}C_{i}^{(\frac{q^m-1}{q-1}, q^{2m})},$$
where $I=\left\{i~|~ {\rm Tr}_{q^m/q}(\omega^i)=0,\ 0\leq i<\frac{q^m-1}{q-1}\right\}$, and $|I|=\frac{q^{m-1}-1}{q-1}$.
\end{lemma}
\begin{proof}
For any integers $s$ and $j$, we have
$$Q\left(\alpha^{(q^m-1)s+j}\right)={\rm Tr}_{q^m/q}\left(\alpha^{(q^m+1)(q^m-1)s+(q^m+1)j}\right)
={\rm Tr}_{q^m/q}\left(\omega^j\right).$$
It follows that
$$\{x \in \mathbb{F}_{q^{2m}}\setminus \{0\}|~ Q(x)=0\}=\bigcup\limits_{j \in J}C_{j}^{(q^m-1, q^{2m})}, $$
where $J=\left\{j~|~ {\rm Tr}_{q^m/q}(\omega^j)=0,\ 0\leq j<q^m-1\right\}$, $|J|=q^{m-1}-1$. Noting that 
\[
\begin{array}{l}
\vspace{0.2cm}J=\left\{j~|~ {\rm Tr}_{q^m/q}(\omega^j)=0,\ 0\leq j<q^m-1\right\}\\
\vspace{0.2cm}\hspace{0.3cm}=\left\{i+\frac{\ell (q^m-1)}{q-1}~|~ {\rm Tr}_{q^m/q}\left(\omega^{i+\frac{\ell (q^m-1)}{q-1}}\right)=0,\ 0\leq i<\frac{q^m-1}{q-1},\ 0\leq \ell <q-1\right\}\\
\vspace{0.2cm}\hspace{0.3cm}=\left\{i+\frac{\ell (q^m-1)}{q-1}~|~ i\in I,\ 0\leq \ell <q-1\right\},
\end{array}
\]
%$$J=\left\{j~|~ Tr_{q^m/q}(\omega^j)=0,\ 0\leq j<q^m-1\right\} =\left\{i+\frac{l(q^m-1)}{q-1}~|~ Tr_{q^m/q}\left(\omega^{i+\frac{l(q^m-1)}{q-1}}\right)=0,\ 0\leq j<\frac{q^m-1}{q-1}\right\},$$ 
we have 
$$\{x \in \mathbb{F}_{q^{2m}}\setminus \{0\}|~ Q(x)=0\}=\bigcup\limits_{j \in J}C_{j}^{(q^m-1, q^{2m})}=\bigcup\limits_{i \in I}C_{i}^{(\frac{q^m-1}{q-1}, q^{2m})}.$$ 
The proof is now complete.  \qed 
\end{proof}

Note that the quadratic form $Q(x)={\rm Tr}_{q^m/q}\left(x^{q^m+1}\right)$ is of elliptic type. Applying Lemmas \ref{PDS1} and \ref{QF1} to this nonsingular quadratic form yields the following result.

\begin{lemma}\label{PDS2} 
Let $Q :\mathbb{F}_{q^{2m}}\rightarrow \mathbb{F}_q$ be given by $Q(x)={\rm Tr}_{q^m/q}\left(x^{q^m+1}\right)$. Let $\omega$ be a primitive element of $\mathbb{F}_{q^{m}}$ and  $I=\left\{i~|~ {\rm Tr}_{q^m/q}(\omega^i)=0,\ 0\leq i<\frac{q^m-1}{q-1}\right\}$. Then 
$$\{x \in \mathbb{F}_{q^{2m}}\setminus \{0\}|~ Q(x)=0\}=\bigcup\limits_{i \in I}C_{i}^{(\frac{q^m-1}{q-1}, q^{2m})}$$
is a regular $(q^{2m}, (q^{m}+1)(q^{m-1}-1), q^{2m-2}-q^{m-1}(q-1)-2, q^{2m-2}-q^{m-1})$-PDS in the additive group of $\mathbb{F}_{q^{2m}}$. In particular, for $b\in \mathbb{F}_{q^{2m}}\setminus \{0\},$
\[
\begin{array}{l}
\varphi_b\left( \bigcup\limits_{i \in I}C_{i}^{(\frac{q^m-1}{q-1}, q^{2m})} \right)=
\left\{\begin{array}{ll}
\vspace{0.2cm}(q^{m-1}-1)-q^m, & {\rm  if} \ \ b\in \bigcup\limits_{i \in I}C_{i}^{(\frac{q^m-1}{q-1}, q^{2m})}, \\
q^{m-1}-1,  & {\rm otherwise}.\
\end{array}
\right .
\end{array}
\]
\end{lemma}

\section{Partial difference sets from quadratic forms and cyclotomy}
Let $\alpha$ be a fixed primitive element of  $\mathbb{F}_{q^{2m}}$. Set $\omega=\alpha^{q^m+1}$, and view $\mathbb{F}_{q^m}$ as an $m$-dimensional vector space over $\mathbb{F}_q$. For $1\leq r<m,$ let $R$ be an $r$-dimensional vector subspace of $\mathbb{F}_{q^{m}}$ over $\mathbb{F}_{q}$. Set
\begin{equation}
\label{dimv}
T=\left\{t~|~\omega^t\in R,\ 0\leq t<\frac{q^m-1}{q-1}\right\}.
\end{equation}
Clearly,  
$R=\left\{\omega^{t+\frac{\ell (q^m-1)}{q-1}}~|~ t\in T,\ 0\leq \ell<q-1 \right\}\cup\{0\}.$ Then we have 
$|T|=\frac{q^r-1}{q-1}.$

\begin{lemma}\label{QF2}
Let $T$ be defined as in (\ref{dimv}). Then for each $u$, $0\leq u<q^{2m}-1$,
$$\left|\left\{ \alpha^{t+u}~|~ t\in T\right\}\cap \left(\bigcup\limits_{i \in I}C_{i}^{(\frac{q^m-1}{q-1}, q^{2m})}\right)\right|=\frac{q^r-1}{q-1}\ {\rm or}\ \frac{q^{r-1}-1}{q-1},$$
where $I=\left\{i~|~ {\rm Tr}_{q^m/q}(\omega^i)=0,\ 0\leq i<\frac{q^m-1}{q-1}\right\}$.
\end{lemma}
\begin{proof}
Since Tr$_{q^m/q}$ is an $\mathbb{F}_q$-linear transformation from $\mathbb{F}_{q^{m}}$ to $\mathbb{F}_{q}$,  the set 
$$\{ x \in \mathbb{F}_{q^{m}}~|~ {\rm Tr}_{q^m/q}(\omega^v x)=0\}$$
 forms an $(m-1)$-dimensional $\mathbb{F}_q$-subspace of $\mathbb{F}_{q^{m}}$ for each $0\leq v<q^{m}-1$. Noting that $R$ is an $r$-dimensional $\mathbb{F}_q$-subspace of $\mathbb{F}_{q^{m}}$, by the dimension formula, we obtain 
$${\rm dim}_{\mathbb{F}_{q}}\left( R \cap \{ x \in \mathbb{F}_{q^{m}} ~|~ {\rm Tr}_{q^m/q}(\omega^v x)=0\}\right)=r\ {\rm or}\ r-1, $$
i.e.
$$\left|\left( R\setminus \{0\} \right)\cap \{ x \in \mathbb{F}_{q^{m}}^*\mid  {\rm Tr}_{q^m/q}(\omega^v x)=0\}\right|=q^r-1\ {\rm or}\ q^{r-1}-1.$$
From the definition of $T$ (that is, viewing $R$ projectively), we have 
%$$\left|\left\{\omega^{t}~|~ t\in T \right\}\cap \{ x \in \mathbb{F}_{q^{m}}\setminus \{0\}~|~ {\rm Tr}_{q^m/q}(\omega^v x)=0\}\right|=\frac{q^r-1}{q-1}\ {\rm or}\ \frac{q^{r-1}-1}{q-1}.$$
\[
\begin{array}{l}
\vspace{0.2cm}\left|\left\{\omega^{t}~|~ t\in T \right\}\cap \{ x \in \mathbb{F}_{q^{m}}^*~|~ {\rm Tr}_{q^m/q}(\omega^v x)=0\}\right|\\
\vspace{0.2cm}=\frac{1}{q-1}\left|\left\{\omega^{t}~|~ t\in R \setminus \{0\} \right\}\cap \{ x \in \mathbb{F}_{q^{m}}^*~|~ {\rm Tr}_{q^m/q}(\omega^v x)=0\}\right|\\
\vspace{0.2cm}=\frac{q^r-1}{q-1}\ {\rm or}\ \frac{q^{r-1}-1}{q-1}.
\end{array}
\]
It follows that
$$\left|\left\{\omega^{t+v}~|~ t\in T \right\}\cap \{ x \in \mathbb{F}_{q^{m}}^*~|~ {\rm Tr}_{q^m/q}(x)=0\}\right|=\frac{q^r-1}{q-1}\ {\rm or}\ \frac{q^{r-1}-1}{q-1},$$
i.e.
$$\left|\left\{t+v \pmod{q^m-1 }~|~ t\in T \right\}\cap J \right|=\frac{q^r-1}{q-1}\ {\rm or}\ \frac{q^{r-1}-1}{q-1},$$
where $J=\{ j~|~ {\rm Tr}_{q^m/q}(\omega^j)=0,\ 0\leq j <q^m-1 \}$ and $0\leq v<q^{m}-1$. 

For each $u$, $0\le u<q^{2m}-1$, define two sets  
$$A:=\left\{t+u \pmod{q^{m}-1 }~|~ t\in T \right\}\cap J$$
and
$$B:=\left\{t+u \pmod{q^{2m}-1 }~|~ t\in T \right\}\cap \{ j+\ell (q^m-1) ~|~ j\in J,\ 0\leq \ell < q^m+1\}.$$  
We claim that $|A|=|B|$.
On the one hand, for any $x\in A$, we have 
$$x=j_0=t_0+u-\ell_0(q^m-1),$$ 
where $j_0\in J$ and $\ell_0$ is a nonnegative integer, implying $x+\ell_0(q^m-1)=j_0+\ell_0(q^m-1)=t_0+u$. Thus, 
$$x+\ell_0(q^m-1)\in B,$$
where $\ell_0$ is uniquely determined by $x$. 
On the other hand, for any  $y\in B,$ we have $$y=j_0+\ell_0(q^m-1)=t_0+u+s(q^{2m}-1).$$ 
It follows that $y\equiv j_0\equiv t_0+u\pmod {q^m-1}$. Note that $0\leq j_0<q^m-1$, we have $j_0\in A$. Suppose that there is some other $z\in B$ such that $z\equiv j_0\pmod {q^m-1}$, i.e.,
 $$z=j_0+\ell_1(q^m-1)=t_1+u+s_1(q^{2m}-1).$$
Then, we deduce that $t_1=t_0+(s-s_1)(q^{2m}-1)+(\ell_1-\ell_0)(q^m-1)$. By the definition of $T$, we get $\ell_0=\ell_1$, which implies $x=y$. Thus $|A|=|B|$; and from which  it follows that 
% \textcolor{blue}{Let $u$ be an integer such that $0\leq u<q^{2m}-1$. For each $u$, it can be uniquely represented as $u=a_u(q^m-1)+b_u$, where $a_u$ and $b_u$ are integers that satisfy $0\leq a_u <q^m+1$ and $0\leq b_u<q^m-1$, respectively. Set $J'=\{ j+\ell (q^m-1) ~|~ j\in J,\ 0\leq \ell < q^m+1\}.$ It follows that 
%\begin{equation}
%\label{dimv1}
%\begin{aligned}
%\vspace{0.2cm}&\left|\left\{t+u \pmod{q^{2m}-1 }~|~ t\in T \right\}\cap \{ j+\ell (q^m-1) ~|~ j\in J,\ 0\leq \ell < q^m+1\}\right|\\
%\vspace{0.2cm}&=\left|\left\{t+b_u+a_u(q^m-1) \pmod{q^{2m}-1 }~|~ t\in T,\ 0\leq t<q^m-1-b_u \right\}\cap J' \right|\\
%\vspace{0.2cm}&\hspace{0.3cm}+\left|\left\{t+b_u-(q^m-1)+(a_u+1)(q^m-1) \pmod{q^{2m}-1 }~|~ t\in T,\ q^m-1-b_u\leq t< \frac{q^m-1}{q-1} \right\}\cap J' \right|\\
%\vspace{0.2cm}&=\left|\left\{t+b_u ~|~ t\in T,\ 0\leq t<q^m-1-b_u \right\}\cap J \right|\\
%\vspace{0.2cm}&\hspace{0.3cm}+\left|\left\{t+b_u-(q^m-1) ~|~ t\in T,\ q^m-1-b_u\leq t< \frac{q^m-1}{q-1} \right\}\cap J \right|\\
%\vspace{0.2cm}&=\left|\left\{t+b_u \pmod{q^{m}-1 }~|~ t\in T \right\}\cap J \right|\\
%\vspace{0.2cm}&=\frac{q^r-1}{q-1}\ {\rm or}\ \frac{q^{r-1}-1}{q-1}.
%\end{aligned}
%\end{equation}
$$\Big|\left\{t+u \pmod{q^{2m}-1 }~|~ t\in T \right\}\cap \{ j+\ell (q^m-1) ~|~ j\in J,\ 0\leq \ell < q^m+1\}\Big|=\frac{q^r-1}{q-1}\ {\rm or}\ \frac{q^{r-1}-1}{q-1}.$$
%For each $i$, $0\leq i<q^{m}-1$, from the definition of $T$,  we have
%\[
%\begin{array}{l}
%\vspace{0.2cm}\left|\left\{t+u \pmod{q^{2m}-1 }~|~ t\in T \right\}\cap \left\{ i+\ell (q^m-1) ~|~ \ 0\leq \ell < q^m+1\right\}\right|\\
%\vspace{0.2cm}=\left|\left\{t+u \pmod{q^{m}-1 }~|~ t\in T \right\}\cap \{ i\}\right|.
%\end{array}
%\]
%It follows that for each $u$, $0\leq u<q^{2m}-1$, we have
%\begin{equation}
%\label{dimv1}
%\begin{aligned}
%\vspace{0.2cm}&\left|\left\{t+u \pmod{q^{2m}-1 }~|~ t\in T \right\}\cap \{ j+\ell (q^m-1) ~|~ j\in J,\ 0\leq \ell < q^m+1\}\right|\\
%\vspace{0.2cm}&=\left|\left\{t+u \pmod{q^{m}-1 }~|~ t\in T \right\}\cap J \right|\\
%\vspace{0.2cm}&=\frac{q^r-1}{q-1}\ {\rm or}\ \frac{q^{r-1}-1}{q-1}.
%\end{aligned}
%\end{equation}
%Since we can rewrite $\bigcup\limits_{i \in I}C_{i}^{(\frac{q^m-1}{q-1}, q^{2m})}$ as 
Observing that
$$\bigcup\limits_{i \in I}C_{i}^{(\frac{q^m-1}{q-1}, q^{2m})}=\bigcup\limits_{j \in J}C_{j}^{(q^m-1, q^{2m})},$$
we have
$$\left|\left\{ \alpha^{t+u}~|~ t\in T\right\}\cap \left(\bigcup\limits_{i \in I}C_{i}^{(\frac{q^m-1}{q-1}, q^{2m})}\right)\right|=\left|\left\{ \alpha^{t+u}~|~ t\in T\right\}\cap \left(\bigcup\limits_{j \in J}C_{j}^{(q^m-1, q^{2m})}\right)\right|=\frac{q^r-1}{q-1}\ {\rm or}\ \frac{q^{r-1}-1}{q-1}$$
for each $u, 0\leq u<q^{2m}-1$. The proof is now complete.  \qed 
\end{proof}

Let $T$ be defined as in (\ref{dimv}). We define the following subset $\D$ of $\mathbb{F}_{q^m}\times \mathbb{F}_{q^{2m}}$. 
$$\D:=\bigcup\limits_{i=0}^{\frac{q^m-q}{q-1}}\left(C_i^{(\frac{q^m-1}{q-1}, q^{m})}\times \bigcup\limits_{t\in T}C_{i+t}^{(\frac{q^m-1}{q-1}, q^{2m})}\right)\cup \left(\mathbb{F}_{q^{m}}^{\ast}\times \{0\}\right).$$
Let $\alpha$ be a fixed primitive element of  $\mathbb{F}_{q^{2m}}$ and set $\omega=\alpha^{q^m+1}$. The subset $\D$ is nothing but

$$\D=\bigcup\limits_{i=0}^{\frac{q^m-q}{q-1}}\left(\omega^i\mathbb{F}_q^{\ast} \times \bigcup\limits_{t\in T}C_{i+t}^{(\frac{q^m-1}{q-1}, q^{2m})}\right)\cup \left(\mathbb{F}_{q^{m}}^{\ast}\times \{0\}\right).$$
%
%The next theorem is the main result in the paper; and it shows that PDS with Denniston parameters exist for all $r, 1\leq r<m.$ 
Now we are ready to prove the main theorem of this paper, which says that $\D$ is a PDS with Denniston parameters in the additive group of $\mathbb{F}_{q^{m}}\times \mathbb{F}_{q^{2m}}$.
\begin{theorem}\label{MT}
The set $\D$ defined above  is a $(q^{3m}, (q^{m+r}-q^m+q^r)(q^m-1), q^m-q^r+(q^{m+r}-q^m+q^r)(q^r-2), (q^{m+r}-q^m+q^r)(q^r-1))$-PDS in the additive group of $\mathbb{F}_{q^{m}}\times \mathbb{F}_{q^{2m}}$. 
\end{theorem}

\begin{proof}
First of all, we compute $|\D|=\frac{q^m-1}{q-1}(q-1)\cdot |T|\cdot (q^m+1)(q-1)+(q^m-1)=(q^{m+r}-q^m+q^r)(q^m-1)$. Every character $\phi$ of the additive group of $\mathbb{F}_{q^{m}}\times \mathbb{F}_{q^{2m}}$ can be written as $\phi=\psi_a\times \varphi_b$, where $\psi_a$ and $\varphi_b$ are additive characters of  $\mathbb{F}_{q^{m}}$ and $\mathbb{F}_{q^{2m}}$, respectively.  Each nonprincipal character $\phi$ falls into one of the following three cases:
 %$\psi_a$ is a character of the additive group of $\mathbb{F}_{q^{m}}$ and $\varphi_b$ is a character of the additive group of $\mathbb{F}_{q^{2m}}$.

 % We consider three cases.
\begin{description}
	\item [Case (i)]  $\psi_a=\psi_0$ and $\varphi_b \neq \varphi_0$. 

In this case, we have $\psi_0\left(\mathbb{F}_{q^{m}}^{\ast}\right)=q^m-1$ and $\psi_0\left(C_i^{(\frac{q^m-1}{q-1}, q^{m})}\right)=\psi_0(\omega^i \mathbb{F}_q^{\ast})=q-1$ for each $0\leq i<\frac{q^m-1}{q-1}$. We calculate $\phi({\cal D})$ as follows:
\[
\begin{array}{l}
\vspace{0.2cm}\phi({\cal D})=\sum\limits_{i=0}^{\frac{q^m-q}{q-1}}\psi_0\left(C_i^{(\frac{q^m-1}{q-1}, q^{m})}\right)\varphi_b\left(\bigcup\limits_{t\in T} C_{i+t}^{(\frac{q^m-1}{q-1}, q^{2m})}\right)+\psi_0\left(\mathbb{F}_{q^{m}}^{\ast}\right)\varphi_b\left(0\right)\\
\vspace{0.2cm}\hspace{0.85cm}=(q-1)\sum\limits_{i=0}^{\frac{q^m-q}{q-1}}\varphi_b\left( \bigcup\limits_{t\in T}C_{i+t}^{(\frac{q^m-1}{q-1}, q^{2m})}\right)+q^m-1\\
\vspace{0.2cm}\hspace{0.85cm}=(q-1)\sum\limits_{t\in T}\varphi_b\left(  \bigcup\limits_{i=0}^{\frac{q^m-q}{q-1}} C_{i+t}^{(\frac{q^m-1}{q-1}, q^{2m})}\right)+q^m-1\\
\vspace{0.2cm}\hspace{0.85cm}=(q-1)\sum\limits_{t\in T}\varphi_b\left( \mathbb{F}_{q^{2m}}^{\ast} \right)+q^m-1\\
%\vspace{0.2cm}\hspace{0.85cm}=(q-1)(-1)+q^m-1\\
\vspace{0.2cm}\hspace{0.85cm}=-(q-1)|T|+q^m-1\\
\vspace{0.2cm}\hspace{0.85cm}=q^m-q^r.
\end{array}
\]

\item [Case(ii)] $\psi_a\neq \psi_0$ and $\varphi_b=\varphi_0$.

 Since $|T|=\frac{q^r-1}{q-1}$, we have $\varphi_0\left(\bigcup\limits_{t\in T} C_{i+t}^{(\frac{q^m-1}{q-1}, q^{2m})}\right)=(q^r-1)(q^m+1)$ for each $0\leq i<\frac{q^m-1}{q-1}$.  We calculate $\phi({\cal D})$ as follows:
\[
\begin{array}{l}
\vspace{0.2cm}\phi({\cal D})=\sum\limits_{i=0}^{\frac{q^m-q}{q-1}}\psi_a\left(C_i^{(\frac{q^m-1}{q-1}, q^{m})}\right)\varphi_0\left(\bigcup\limits_{t\in T} C_{i+t}^{(\frac{q^m-1}{q-1}, q^{2m})}\right)+\psi_a\left(\mathbb{F}_{q^{m}}^{\ast}\right)\varphi_0\left(0\right)\\
\vspace{0.2cm}\hspace{0.85cm}=(q^r-1)(q^m+1)\sum\limits_{i=0}^{\frac{q^m-q}{q-1}}\psi_a\left( C_i^{(\frac{q^m-1}{q-1}, q^{m})}\right)-1\\
%\vspace{0.2cm}\hspace{0.85cm}=(q-1)(-1)+q^m-1\\
\vspace{0.2cm}\hspace{0.85cm}=-q^{m+r}+q^m-q^r.
\end{array}
\]

\item [Case (iii)] $\psi_a \neq \psi_0$ and $\varphi_b \neq \varphi_0$.

 Let $a=\omega^c$ for some $0\le c<q^m-1$. For each $0\leq i< \frac{q^m-1}{q-1}$, let $A_{i}$ be the subset of the character group ${\widehat{\mathbb{F}}_{q^m}}$ consisting of all characters which are principal on $C_i^{(\frac{q^m-1}{q-1}, q^{m})}=\omega^i\mathbb{F}_q^{\ast}$. Then 
\begin{equation}
\label{TR}
\begin{aligned}
\vspace{0.2cm}A_{i}&=\left\{ \psi_{\omega^c} \in {\widehat{\mathbb{F}}_{q^m}}\mid {\rm Tr}_{q^m/p}\left(\omega^{c+i}s\right)=0,\;\forall s\in \mathbb{F}_q\right\}\cup\{\chi_0\}\\
\vspace{0.2cm}\hspace{0.45cm}&=\left\{ \psi_{\omega^c} \in {\widehat{\mathbb{F}}_{q^m}}\mid  {\rm Tr}_{q^m/q}\left(\omega^{c+i}\right)=0\right\}\cup\{\chi_0\}.
%\vspace{0.2cm}\hspace{0.45cm}&=\left\{\chi_{\omega^{j-i+\frac{l(q^m-1)}{q-1}}} \in \widehat{(\mathbb{F}_{q^m},+)}|~ j\in B, 0\leq l<q-1 \right\}\cup\{\chi_0\}.
\end{aligned}
\end{equation}

For each $0\leq i<\frac{q^m-1}{q-1}$, we have  
\[
\begin{array}{l}
\psi_a\left(C_i^{(\frac{q^m-1}{q-1}, q^{m})}\right)=
\left\{\begin{array}{ll}
\vspace{0.2cm}q-1, & {\rm  if} \ \ \psi_a \in A_i, \\
-1,  & {\rm otherwise}.\
\end{array}
\right .
\end{array}
\]
For each $s\in \mathbb{F}_{q^m}^{\ast}$, let $B_s=\left\{i~|~\psi_s \in A_i, 0\leq i<\frac{q^m-1}{q-1}\right\}$. By (\ref{TR}), we have 
\begin{equation}
\label{TR1}
\begin{aligned}
B_a=\left\{j-c~|~ j\in I \right\}\ {\rm and}\ 
\psi_a\left(C_i^{(\frac{q^m-1}{q-1}, q^{m})}\right)=
\left\{\begin{array}{ll}
\vspace{0.2cm}q-1, & {\rm  if} \ \ i \in B_a, \\
-1,  & {\rm otherwise},\
\end{array}
\right .
\end{aligned}
\end{equation}
where $a=\omega^c$ and $I=\left\{i~|~ {\rm Tr}_{q^m/q}(\omega^i)=0,\ 0\leq i<\frac{q^m-1}{q-1}\right\}.$
Now we calculate $\phi({\cal D})$ as follows:
%. By (\ref{TR1}), we have
\[
\begin{array}{l}
\vspace{0.2cm}\phi({\cal D})=\sum\limits_{i=0}^{\frac{q^m-q}{q-1}}\psi_a\left(\omega^i\mathbb{F}_q^*\right)\varphi_b\left(\bigcup\limits_{t\in T} C_{i+t}^{(\frac{q^m-1}{q-1}, q^{2m})}\right)+\psi_a\left(\mathbb{F}_{q^{m}}^{\ast}\right)\varphi_b\left(0\right)\\
\vspace{0.2cm}\hspace{0.85cm}=(q-1)\sum\limits_{i\in B_a}\varphi_b\left(\bigcup\limits_{t\in T} C_{i+t}^{(\frac{q^m-1}{q-1}, q^{2m})}\right)-\sum\limits_{i\notin B_a}\varphi_b\left(\bigcup\limits_{t\in T} C_{i+t}^{(\frac{q^m-1}{q-1}, q^{2m})}\right)-1\\
\vspace{0.2cm}\hspace{0.85cm}=q\sum\limits_{i\in B_a}\varphi_b\left(\bigcup\limits_{t\in T} C_{i+t}^{(\frac{q^m-1}{q-1}, q^{2m})}\right)+|T|-1\\
\vspace{0.2cm}\hspace{0.85cm}=q\sum\limits_{i \in I}\sum\limits_{t\in T}\varphi_b\left(C_{i+t-c}^{(\frac{q^m-1}{q-1}, q^{2m})}\right)+\frac{q^r-q}{q-1}\\
\vspace{0.2cm}\hspace{0.85cm}=q\sum\limits_{t\in T}\varphi_{b\alpha^{t-c}}\left(\bigcup\limits_{i \in I}C_{i}^{(\frac{q^m-1}{q-1}, q^{2m})}\right)+\frac{q^r-q}{q-1}.
\end{array}
\]
By Lemma \ref{QF2}, we know that
$$\left|\left\{ b\alpha^{t-c}~|~ t\in T\right\}\cap \left(\bigcup\limits_{i \in I}C_{i}^{(\frac{q^m-1}{q-1}, q^{2m})}\right)\right|=\frac{q^r-1}{q-1}\ {\rm or}\ \frac{q^{r-1}-1}{q-1}$$
for any $b\in \mathbb{F}_{q^{2m}}\setminus\{0\}$, where $I=\left\{i~|~ {\rm Tr}_{q^m/q}(\omega^i)=0,\ 0\leq i<\frac{q^m-1}{q-1}\right\}$. 
Applying Lemma \ref{PDS2}, we obtain 
\[
\begin{array}{l}
\vspace{0.2cm}\phi({\cal D})=q\sum\limits_{t\in T}\varphi_{b\alpha^{t-c}}\left(\bigcup\limits_{i \in I}C_{i}^{(\frac{q^m-1}{q-1}, q^{2m})}\right)+\frac{q^r-q}{q-1}\\
\vspace{0.2cm}\hspace{0.85cm}=
\left\{\begin{array}{ll}
\vspace{0.2cm}q(q^{m-1}-q^{m}-1)\frac{q^r-1}{q-1}+\frac{q^r-q}{q-1}, \; \;\;\;{\rm  if} \ \left\{ b\alpha^{t-c}~|~ t\in T\right\} \subset \bigcup\limits_{i \in I}C_{i}^{(\frac{q^m-1}{q-1}, q^{2m})}, \\
q(q^{m-1}-q^{m}-1)\frac{q^{r-1}-1}{q-1}+q(q^{m-1}-1)q^{r-1}+\frac{q^r-q}{q-1},  \;\;\; {\rm otherwise},\\
\end{array}
\right .\\
\vspace{0.2cm}\hspace{0.85cm}=
\left\{\begin{array}{ll}
\vspace{0.2cm}-q^{m+r}+q^{m}-q^r, & {\rm  if} \ \left\{ b\alpha^{t-c}~|~ t\in T\right\} \subset \bigcup\limits_{i \in I}C_{i}^{(\frac{q^m-1}{q-1}, q^{2m})}, \\
q^{m}-q^{r},  & {\rm otherwise}.\
\end{array}
\right .
\end{array}
\]
\end{description}
Thus, for every nonprincipal character $\phi$, we have shown that 
$$\phi(\D)=q^{m}-q^{r}\  \mathrm{or} \ -q^{m+r}+q^{m}-q^r,$$
It follows from Theorem \ref{PDS} that $\D$ is a PDS. The proof is now complete.  \qed 

\remark
The PDS $\D$ in the above theorem is $\mathbb{F}_q^*$-invariant, but not $\mathbb{F}_{q^m}^*$-invariant. This can be seen as follows. 
For any $k$, $1\leq k<q^m-1$, we have 
\[
\begin{array}{l}
\vspace{0.2cm} \omega^k{\cal D}=\bigcup\limits_{i=0}^{\frac{q^m-q}{q-1}}\left[\left(\omega^kC_i^{(\frac{q^m-1}{q-1}, q^{m})}\right)\times \left( \bigcup\limits_{t\in T}\omega^kC_{i+t}^{(\frac{q^m-1}{q-1}, q^{2m})}\right)\right]\cup \left(\omega^k\mathbb{F}_{q^{m}}^{\ast}\times \{0\}\right)\\
\vspace{0.2cm}\hspace{0.8cm}=\bigcup\limits_{i=0}^{\frac{q^m-q}{q-1}}\left[C_{i+k}^{(\frac{q^m-1}{q-1}, q^{m})}\times \left( \bigcup\limits_{t\in T}\alpha^{k(q^m+1)}C_{i+t}^{(\frac{q^m-1}{q-1}, q^{2m})}\right)\right]\cup \left(\mathbb{F}_{q^{m}}^{\ast}\times \{0\}\right).
\end{array}
\]
Since $q^m+1\equiv 2\pmod{\frac{q^m-1}{q-1}}$ and $\alpha$ is a primitive element of $\mathbb{F}_{q^{2m}}$, we have
\[
\begin{array}{l}
\vspace{0.2cm} \omega^k{\cal D}=\bigcup\limits_{i=0}^{\frac{q^m-q}{q-1}}\left[C_{i+k}^{(\frac{q^m-1}{q-1}, q^{m})}\times \bigcup\limits_{t\in T}C_{i+t+2k}^{(\frac{q^m-1}{q-1}, q^{2m})}\right]\cup \left(\mathbb{F}_{q^{m}}^{\ast}\times \{0\}\right)\\
\vspace{0.2cm}\hspace{0.8cm}=\bigcup\limits_{i=0}^{\frac{q^m-q}{q-1}}\left[C_{i}^{(\frac{q^m-1}{q-1}, q^{m})}\times \bigcup\limits_{t\in T}C_{i+t+k}^{(\frac{q^m-1}{q-1}, q^{2m})}\right]\cup \left(\mathbb{F}_{q^{m}}^{\ast}\times \{0\}\right)\\
\vspace{0.2cm}\hspace{0.8cm}\neq \D.
\end{array}
\]
This shows that $\D$ does not arise from a projective $\left(q^{m+r}-q^m+q^r,3,h_1,h_2\right)$ set in ${\rm PG}(2,q^m)$.

\end{proof}

\begin{corollary}\label{srgexistence}
Let $m\ge 2$ be a positive integer. For any $1\le r<m-1$, and any prime power $q$, there exists a strongly regular Cayley graph with parameters $(q^{3m}, (q^{m+r}-q^m+q^r)(q^m-1), q^m-q^r+(q^{m+r}-q^m+q^r)(q^r-2), (q^{m+r}-q^m+q^r)(q^r-1))$.
\end{corollary}

\section{Conclusions}
In this paper, we have constructed PDS in elementary abelian groups of order $q^{3m}$, for all $m\ge 2$ and $1\le r<m$, where $q$ is a prime power. It is natural to ask whether one can construct PDS with Denniston parameters in non-elementary abelian $p$-groups, or in nonabelian $p$-groups. Some work has been done in this direction. For example, Davis and Xiang \cite{DX2000} used Galois rings to construct PDS with the Denniston parameters in $\mathbb{Z}_{4}^{m}\times \mathbb{Z}_{2}^{m}$ for all $m\geq 2$ and $r\in \{1,m-1\}$; Brady \cite{B2022} did a computer search to find examples in the $m=2, r=1$ case for $51$ groups of order $64$. In view of the result in \cite{DX2000}, it is natural to ask whether there exists PDS with Denniston parameters in non-elementary abelian $p$-groups when $p$ is an odd prime.

\vspace{0.1in}

\noindent{\bf Acknowledgement.} The research work of Qing Xiang is partially supported by the National Natural Science Foundation of China Grant No. 12071206, 12131011, 12150710510.


\begin{thebibliography}{Z}
\baselineskip 10pt

\bibitem{BBM1997}
S. Ball, A. Blokhuis, F. Mazzocca, Maximal arcs in Desarguesian planes of odd order do not exist, {\it Combinatorica} {\bf 17} (1997), 31-41.

\bibitem{B2022}
A. C. Brady, Negative Latin square type partial difference sets in nonabelian groups of order $64$, {\it Finite Fields Appl.} {\bf 81} (2022) Paper No. 102044, 11pp.

%\bibitem{B1985}
%A.E. Brouwer, Some new two-weight codes and strongly regular graphs, Discrete Appl. Math. 10 (1985) 111-114.

\bibitem{BM2022}
A. E. Brouwer, H. Van Maldeghem, {\it Strongly Regular Graphs}, Cambridge University Press, 2022.

%\bibitem{BWX1999}
%A.E. Brouwer, R.M. Wilson, Q. Xiang, Cyclotomy and strongly regular graphs, J. Algebraic Comb. 10 (1999) 25-28.

\bibitem{CK1986}
R. Calderbank, W. M. Kantor, The geometry of two-weight codes, {\it Bull. Lond. Math. Soc.} {\bf 18} (1986), 97-122.

\bibitem{Cameron}
P. J. Cameron, Finite geometry and coding theory, Lecture Notes for Socrates Intensive Programme “Finite Geometries and Their Automorphisms”, Potenza, Italy, June 1999. 

\bibitem{CLeo}
A. Cossidente, L. Storme, Caps on elliptic quadrics, {\em Finite Fields Appl.} {\bf 1} (1995), 412-420. 

\bibitem{CRX1996}
Y. Chen, D. K. Ray-Chaudhuri, Q. Xiang, Constructions of partial difference sets and relative difference sets using galois rings II, {\em J. Comb. Theory Ser. A}  {\bf 76} (1996), 179-196.

\bibitem{DHJP2024}
J. A. Davis, S. Huczynska, L. Johnson, J. Polhill, Denniston partial difference sets exist in the odd prime case, {\tt arXiv:2311.00512}.

\bibitem{DX2000}
J. A. Davis, Q. Xiang, A family of partial difference sets with Denniston parameters in nonelementary abelian $2$-groups, {\em Eurp. J. Comb.} {\bf 21} (2000), 981-988.


\bibitem{D1969}
R. Denniston, Some maximal arcs in finite projective planes, {\em J. Comb. Theory} {\bf 6} (1969), 317-319.

\bibitem{dewinter23}
S. De Winter, Two-weight sets of Denniston type, {\tt arXiv:2311.00827v1}.

%\bibitem{FWXY2013}
%T. Feng, B. Wen, Q. Xiang, J. Yin, Partial difference sets from quadratic forms and  $p$-ary weakly regular bent functions, in: Proceedings of a Conference in Honor of the $70$th Birthday of Keqin Feng, Number Theory and Related Areas, in: Adv. Lect. Math., vol. 27, 2013, 25-40.

\bibitem{LM1990}
K. H. Leung, S. L. Ma, Constructions of partial difference sets and relative difference sets on $p$-groups, {\em Bull. Lond. Math. Soc.} {\bf 22} (1990), 533-539.

%\bibitem{LN1997}
%R. Lidl, H. Niederreiter, Finite Fields, Cambridge University Press, 1997.

\bibitem{M1994}
S. L. Ma, A survey of partial difference sets, {\em Des. Codes Cryptogr.} {\bf 4} (1994), 221-261.

%\bibitem{T1965}
%R.J. Turyn, Character sums and difference sets, Pacific J. Math. 15 (1965) 319-346.

%\bibitem{V2003}
%E.R. van Dam, Strongly regular decompositions of the complete graph, J. Algebraic Comb. 17 (2003) 181-201.

\bibitem{vlWilson}
J. H. van Lint, R. M. Wilson, {\em A Course in Combinatorics}, Second Edition, Cambridge University Press, 2001.

\bibitem{W1997}
R. M. Wilson, Q. Xiang, Constructions of Hadamard difference sets, {\em J. Comb. Theory Ser. A}  {\bf 77} (1997), 148-160.


\end{thebibliography}
\end{document}